\documentclass[12pt,letterpaper]{article}

\usepackage[utf8]{inputenc}

\usepackage{amssymb,amsmath,amsthm,amsfonts}
\usepackage[pdftex,colorlinks=true,linkcolor=blue,citecolor=blue,urlcolor=red,unicode=true,hyperfootnotes=true,bookmarksnumbered]{hyperref}
\usepackage[margin=1truein]{geometry}

\usepackage{enumerate}

\theoremstyle{plain}
\newtheorem{theorem}{Theorem}[section]
\newtheorem{lemma}[theorem]{Lemma}
\newtheorem{claim}[theorem]{Claim}

\newtheorem{corollary}[theorem]{Corollary}

\theoremstyle{definition} 

\newtheorem{remark}[theorem]{Remark}

\makeatletter
\def\@gifnextchar#1#2#3{\let\@tempe#1\def\@tempa{#2}\def\@tempb{#3}%
  \futurelet\@tempc\@gifnch}

\def\@gifnch{\ifx\@tempc\@sptoken\let\@tempd\@tempb%
  \else\ifx\@tempc\@tempe\let\@tempd\@tempa\else\let\@tempd\@tempb\fi\fi\@tempd}

\def\SK@set#1{\left\{#1\right\}}
\def\SK@@set#1#2{\{#1\,:\,
    \begin{array}{@{}l@{}}#2\end{array}
\}}
\def\SK@mset#1{\left\{\!\!\left\{#1\right\}\!\!\right\}}
\def\SK@@mset#1#2{\{\!\!\{#1\,:\,
    \begin{array}{@{}l@{}}#2\end{array}
\}\!\!\}}
\def\BIG@set#1{\Big\{#1\Big\}}
\def\BIG@@set#1#2{\Big\{#1\:\Big|\:
    \begin{array}{@{}l@{}}#2\end{array}
\Big\}}
\newcommand{\Set}[1]{\@gifnextchar\bgroup{\SK@@set{#1}}{\SK@set{#1}}}
\newcommand{\Mset}[1]{\@gifnextchar\bgroup{\SK@@mset{#1}}{\SK@mset{#1}}}
\newcommand{\Bigset}[1]{\@gifnextchar\bgroup{\BIG@@set{#1}}{\BIG@set{#1}}}
\makeatother

\title{Weak saturation rank:\\ a failure of linear algebraic approach to weak saturation}

\author{Nikolai Terekhov\thanks{Department of Discrete Mathematics, Moscow Institute of Physics and Technology, Dolgoprudny, Russia; nikolayterek@gmail.com
}, \quad Maksim Zhukovskii\thanks{Department of Computer Science, University of Sheffield, Sheffield S1 4DP, UK; zhukmax@gmail.com}
}

\date{}

\begin{document}

\maketitle

\begin{abstract}
Given a graph $F$ and a positive integer $n$, the weak $F$-saturation number $\mathrm{wsat}(K_n,F)$ is the minimum number of edges in a graph $H$ on $n$ vertices such that the edges missing in $H$ can be added, one at a time, so that every edge creates a copy of $F$. Kalai in 1985 introduced a linear algebraic approach that became one of the most efficient tools to prove lower bounds on weak saturation numbers. If $W$ is a vector space spanned by vectors $w(e)$ assigned to edges $e$ of $K_n$ in such a way that, for every copy $F'\subset K_n$ of $F$, there exist non-zero $\lambda_e$, $e\in E(F')$, satisfying $\sum_{e\in E(F')}\lambda_e w(e)=0$, then $\mathrm{dim}W\leq \mathrm{wsat}(K_n,F)$. In this paper, we prove limitations of this approach: we show infinitely many $F$ such that, for every vector space $W$ as above, $\mathrm{dim}W<\mathrm{wsat}(K_n,F)$. We also suggest a modification of this approach that allows to get tight lower bounds even when the original linear algebraic approach is not sufficient. Finally, we generalise our results to random graphs, complete multipartite graphs, and hypergraphs.
\end{abstract}

\section{Introduction}
\label{sc:intro}

Cellular automata, which were introduced by von Neumann~\cite{Neumann} after a suggestion of Ulam \cite{Ulam}, these days describe and simulate a variety of processes in physics, biology, chemistry, cryptography, and other fields. Bollob\'{a}s \cite{bol} introduced an extensively studied graph bootstrap percolation, which is a particular case of monotone cellular automata, and is a substantial generalisation of $r$-neighborhood bootstrap percolation model having applications in physics; see, for example, \cite{Adler, Fontes,Morris}. Given a graph $F$, an {\it $F$-bootstrap percolation process} is a sequence of graphs $H_0\subset H_1\subset\cdots\subset H_m$ such that, for   $i=1,\ldots,m$, $H_i$ is obtained from $H_{i-1}$ by adding an edge that belongs to a copy of $F$ in $H_i$. A graph $H$ is {\it weakly $F$-saturated} in graph $G$, if there exists an $F$-bootstrap percolation process $H=H_0\subset H_1\subset\cdots\subset H_m=G$. The  minimum number of edges in a weakly $F$-saturated graph in $G$ is called the {\it weak $F$-saturation number} of $G$ and is denoted by $\mathrm{wsat}(G,F)$.

The exact value of $\mathrm{wsat}(G,F)$ is known only in some specific cases. Even when both $G$ and $F$ are cliques --- exactly the situation considered by Bollob\'{a}s in~\cite{bol} --- it took years to prove~\cite{Alon85,Frankl82,Kalai85,Kalai84} the conjecture of Bollob\'{a}s that $\mathrm{wsat}(K_n,K_s)={n\choose 2}-{n-s+2\choose 2}$ (here $K_n$ denotes a complete graph on $n$ vertices). All known proofs of this fact are linear algebraic, and a purely combinatorial proof was given by Bollob\'{a}s only for $s\leq 6$. Since then, linear algebra was applied for a few other pairs $(G,F)$, and it is certainly one of the most efficient tools, although for $F$ with weak connectivity properties and $G=K_n$ combinatorial approaches are known~\cite{TZ}.


In~\cite{Kalai85}, Kalai suggested a general method of finding $\mathrm{wsat}(G,F)$ using matroids on edges of $G$. 
 For a matroid $M=(E,\mathcal{I})$ and a subset $A\subseteq E$, we denote by $\mathrm{rk}_M(A)$, the rank of the submatroid of $M$ induced on $A$. Let us call a matroid $M=(E(G),\mathcal{I})$ on the edges of $G$ {\it weakly $F$-saturated}, if every copy $\tilde F$ of $F$ in $G$ is a {\it cycle} in $M$, i.e.
\begin{equation}
\forall e\in E(\tilde F) \quad \mathrm{rk}_M(E(\tilde F)\setminus e)=\mathrm{rk}_M(E(\tilde F)).
\label{eq:sat-rank-def}
\end{equation}
The {\it weak $F$-saturation rank} of $G$ is the maximum rank of a weakly $F$-saturated matroid $M$ on the edges of $G$, we denote it by $\mathrm{rk}$-$\mathrm{sat}(G,F)$.

\begin{lemma}[\cite{Kalai85}]
For any graphs $G,F$, $\mathrm{wsat}(G,F)\geq\mathrm{rk}$-$\mathrm{sat}(G,F)$.
\label{lm:Kalai}
\end{lemma}

Note that in most known applications, linear matroids are sufficient to get the exact value of $\mathrm{wsat}(G,F)$ from Lemma~\ref{lm:Kalai} (that is the rank is maximised over all linear matroids, and then it appears to coincide with $\mathrm{wsat}(G,F)$). 
 This linear algebraic approach, which is a stronger version of the Kalai's approach, was presented in~\cite{BBMR} as a lemma that asserts a lower bound on $\mathrm{wsat}(G,F)$ in terms of dimension of a vector subspace spanned by edges of $G$ considered as vectors in a certain domain vector space (actually, in~\cite{BBMR} the lemma is formulated in terms of a {\it vertex} bootstrap percolation, and reformulations for the {\it edge} process appeared in~\cite{KMM,Pikhurko}). For $G=K_n$, a standard choice for the ground vector space is an image of a quotient algebra of the tensor algebra $T^2\mathbb{R}^n$ under a linear map that is defined in a special way depending on the choice of $F$. Here, every vertex $i$ of $G$ is represented by a vector $e_i$ in an orthonormal basis in $\mathbb{R}^n$ and every edge $\{i,j\}$ is represented by $e_i\otimes e_j$ in $T^2\mathbb{R}^n$ (note that for any vector space $W$ over a field $\mathbb{F}$ and any map $\varphi:E(G)\to W$, there exists a quotient algebra $W'$ of $T^2\mathbb{F}^n$ with the natural epimorphism $\varphi':T^2\mathbb{F}^n\to W'$ so that $\varphi(\{i,j\})\leftrightarrow\varphi'(e_i\otimes e_j)$ is an isomorphism of the vector spaces spanned by $\varphi(\{i,j\}),\,\{i,j\}\in E(K_n),$ and $\varphi'(e_i\otimes e_j),\,\{i,j\}\in E(K_n)$, respectively). In particular, exterior algebras were used in~\cite{Kalai85} and symmetric algebras were used in~\cite{Kalai85,KMM} to construct the desired matroids.

Lemma~\ref{lm:Kalai} was exploited to find the exact values of $\mathrm{wsat}(G,F)$ when $G=K_n$, $F=K_s$ or $F=K_{s,s}$ in~\cite{Kalai85} ($K_{s,s}$ denotes complete bipartite graph with both parts of size $s$); $G=K_n$ and $F=K_{s,s+1}$ in~\cite{KMM}; $G=K_n$ and $F=K_s\setminus K_t$ in~\cite{Pikhurko}; $G=K_n$ and $F$ is a graph obtained from a $K_s$ and a graph from a certain family by adding all edges between them~\cite{PikhurkoPhD}; 
 $G$ and $F$ are grids (with an additional requirement that pattern grids are axis aligned with the host grid) in~\cite{MNS}. In all these cases it turns out that $\mathrm{wsat}(G,F)=\mathrm{rk}$-$\mathrm{sat}(G,F)$. Lemma~\ref{lm:Kalai} can be also used to find asymptotics of  $\mathrm{wsat}(G,F)$ when  $G=K_n$ and $F$ is a disjoint union of all non-isomorphic graphs with a fixed number of vertices and fixed number of edges (see~\cite{PikhurkoPhD}) and when $G=K_{n,n}$, $F=K_{s,t}$ (see~\cite{KMM}).

Pikhurko in~\cite{PikhurkoPhD} observed that some specific ways of constructing matroids (namely, count matroids) do not allow to get the exact value of $\mathrm{wsat}(K_n,F)$ for certain $F$ using Lemma~\ref{lm:Kalai} since ranks of these matroids are strictly less than $\mathrm{wsat}(K_n,F)$. He also suspected that gross matroids cannot be used to get the exact value of $\mathrm{wsat}(K_n,C_{\ell})$ for even cycles. In this paper, we prove that for infinitely many connected graphs $F$, there exists $\varepsilon=\varepsilon(F)>0$ such that, for large enough $n\geq n_0(F)$, and $G=K_n$, 
\begin{equation}
\mathrm{rk}\text{-}\mathrm{sat}(G,F)<(1-\varepsilon)\mathrm{wsat}(G,F),
\label{eq:fails}
\end{equation}
that is Lemma~\ref{lm:Kalai} cannot be used to get even the asympotics (in $n$) of  $\mathrm{wsat}(K_n,F)$. We prove this in Section~\ref{sc:sat-rank} by showing that, for any $F$, $\lim_{n\to\infty}\mathrm{rk}$-$\mathrm{sat}(K_n,F)/n=a_F$, where $a_F$ is a non-negative {\it integer} that depends only on $F$. On the other hand, there are graphs $F$ such that $\lim_{n\to\infty}\mathrm{wsat}(K_n,F)/n$ is not an integer (see, e.g.,~\cite{TZ}).

In~\cite{Alon85}, Alon proved that $\lim_{n\to\infty}\mathrm{wsat}(K_n,F)/n$ exists for any $F$. Later Tuza conjectured~\cite{Tuza} that, for every $F$, $\mathrm{wsat}(K_n,F)=c_Fn+O(1)$. The conjecture is known to be true for all connected graphs with minimum degree $\delta(F)=2$ and when $c_F=\delta(F)-1$ (say, for $F=K_s$ or $F=K_{s,t}$), see~\cite{TZ}. In Section~\ref{sc:sat-rank}, we prove the reformulation of Tuza's conjecture in terms of the weak saturation rank: $\mathrm{rk}$-$\mathrm{sat}(K_n,F)=a_Fn+O(1)$. It is worth noting that our proof strategy can be also applied to linear matroids. That is,  the maximum rank of a weakly $F$-saturated linear matroid $E(K_n)$ equals $\tilde a_Fn+O(1)$ for some $\tilde a_F\in\mathbb{Z}_{\geq 0}$. Indeed, the only crucial requirement on the family of matroids that we need is to be `subset-closed' (if $M=(E(G=K_n),\mathcal{I})$ is in the family and $U\subset V(G)$, then the matroid $M'=(E(G[U]),\mathcal{I}\cap U)$ induced on $U$ also belongs to the family), see Section~\ref{rk:generalisation_main_theorem}. 

Further, in Section~\ref{sc:new-method}, we introduce a new method that allows to apply Lemma~\ref{lm:Kalai} even when $c_F$ is not an integer by considering matroids on edges of multigraphs. This appears to be possible due to the following simple observation. Let $K_n^k$ be a multigraph version of $K_n$ with every edge having multiplicity $k$, and let $F_e^k$ be obtained from $F$ and $e\in E(F)$ by including each edge from $E(F)\setminus e$ exactly $k$ times and the edge $e$ --- once. Then 
$$
\mathrm{wsat}\left(K_n^k,\left\{F_e^k,\,e\in E(F)\right\}\right)\leq k\cdot\mathrm{wsat}(K_n,F)
$$ 
(the concept of weak saturation in multigraphs for a family of pattern graphs is naturally extended from the weak saturation in graphs --- see the formal definition in Section~\ref{sc:new-method}).

In Section~\ref{sc:random} we show that, for infinitely many connected graphs $F$, \eqref{eq:fails} holds for {\it typical} $G$ as well (i.e. with asymptotical probability 1 for the graph $G\sim G_{n,1/2}$ chosen uniformly at random from the set of all graphs on $[n]:=\{1,\ldots,n\}$). Finally, in Section~\ref{sc:hyper} we show that Lemma~\ref{lm:Kalai} fails to give a tight bound for other discrete structures that frequently appeared in the literature in the context of weak saturation: for directed graphs, multipartite graphs, and hypergraphs.

In Section~\ref{sc:discussions} we discuss possible future avenues and several open questions.










\section{Weak saturation rank of graphs}
\label{sc:sat-rank}

Here we state and prove the following theorem describing the behaviour of $\mathrm{rk}$-$\mathrm{sat}(K_n,F)$ for large enough $n$.

\begin{theorem}\label{th:asat-integer-1}
For every graph $F$, there exist integer constants $N_F,a_F,C_F$ that depend only on $F$ such that, for every $n\geq N_F$, 
$$
\mathrm{rk}\text{-}\mathrm{sat}(K_n,F)=a_Fn + C_F.
$$
\end{theorem}

\begin{proof}
We will use the following straightforward observation.
\begin{claim}
For any two graphs $G$ and $F$, if a matroid $M=(E(G),I)$ is weakly $F$-saturated and $H$ is weakly $F$-saturated in $G$, then $\mathrm{rk}_M(E(H))=\mathrm{rk}_M(E(G))$.
\label{cl:satgraph-preserve-rank}
\end{claim}
Let $v=|V(F)|$. Fix $n$ and assume that $M_n =\left({[n]\choose 2},\mathcal{I}_n\right)$ is a weakly $F$-saturated matroid with the maximum rank $f(n):=\mathrm{rk}$-$\mathrm{sat}(K_n,F)$. Let us note that the deletion of vertices preserve the equality~\eqref{eq:sat-rank-def} implying 
\begin{equation}
 \mathrm{rk}\text{-}\mathrm{sat}(K_{n-1},F)\geq\mathrm{rk}_{M_n}(E(K_n\setminus u))
\label{eq:asat_delete_vertex}
\end{equation} 
for every $u\in[n].$ It worths mentioning that an analogue of this fact is not necessarily true for the weak saturation: the deletion of vertices from a weakly saturated graph does not preserve the property of being weakly saturated, 
  and this is the main reason for different asymptotic behaviours of the weak saturation number and the weak saturation rank for certain graphs~$F$.

Let us consider the graph $H$ on $[n]$ containing all edges adjacent to a vertex from $[v]$ (i.e. $H\cong K_{1,\ldots,1,n-v}$). Clearly $H$ is weakly $F$-saturated in $K_n$. By Claim~\ref{cl:satgraph-preserve-rank}, $M_n$ has a basis $B\subset E(H)$. Since every edge of $H$ contains at most 1 vertex from $[n]\setminus[v]$ and $|B|=f(n)$, there exists $u\in[n]\setminus[v]$ that belongs to at most $\frac{f(n)}{n - v}$ edges from $B$. Then $\mathrm{rk}_{M_n}(K_n\setminus u)\geq f(n)-\left\lfloor\frac{f(n)}{n - v}\right\rfloor$. Due to (\ref{eq:asat_delete_vertex}), we get
\begin{equation}
  f(n-1) 
  \geq f(n) - \left\lfloor\frac{f(n)}{n - v}\right\rfloor.
\label{eq:th1-limit}
\end{equation}
Let us show that this is sufficient to derive the statement of Theorem~\ref{th:asat-integer-1}. The following claim completes the proof.

\begin{claim}
If a sequence $f(n)\in\mathbb{Z}_{\geq 0}$, $n\in\mathbb{Z}_{>0}$, satisfies~\eqref{eq:th1-limit} for all large enough $n$, then there exists an integer $a$ and an integer $C$ such that $f(n)=an+C$ for all large enough $n$.
\label{eq:general_concluding_claim}
\end{claim}

\begin{proof}
Since $f(n) - \left\lfloor f(n)/(n - v)\right\rfloor\ge (n-v-1)\cdot (f(n)/(n-v))$, the inequality
\eqref{eq:th1-limit} implies $f(n-1)/(n - v - 1)\ge f(n)/(n - v)$. Therefore, there exists
$$
a:=\lim_{n\to +\infty}\left\lfloor\frac{f(n)}{n - v}\right\rfloor.
$$
This $a$ is exactly the constant asserted by Claim~\ref{eq:general_concluding_claim}: for sufficiently large $n$, $f(n)=a(n-v)+r(n)$, where $0\leq r(n)<n-v$. Substituting this equality into~\eqref{eq:th1-limit}, we get
$$
r(n-1)\geq r(n)-\left\lfloor\frac{r(n)}{n-v}\right\rfloor=r(n).
$$
This is only possible if there exists $r\in\mathbb{Z}_{\geq 0}$ such that, for all large enough $n$, $r(n)=r$, completing the proof.
\end{proof}
\end{proof}

Let us recall that in~\cite{TZ}, it was proven that for every odd integer $k\geq 3$, there exists a connected graph $F$ with $\mathrm{wsat}(K_n,F)=\frac{k}{2}n+O(1)$. From Theorem~\ref{th:asat-integer-1}, we immediately get that, for infinitely many connected $F$, $\lim_{n\to\infty}\frac{\mathrm{wsat}(K_n,F)}{n}\neq\lim_{n\to\infty}\frac{\mathrm{rk}\text{-}\mathrm{sat}(K_n,F)}{n}$.

\subsection{Generalisation to subset-closed families of matroids.}
\label{rk:generalisation_main_theorem} 
Let $\mathcal{M}$ be a {\it subset-closed} family of matroids, i.e., for every matroid $M=(E(K_n),\mathcal{I})$ from $\mathcal{M}$ and for every $U\subset V(K_n)$, the induced matroid $(E(K_n[U]),\mathcal{I}\cap U)$ belongs to $\mathcal{M}$. Let $\mathrm{rk}$-$\mathrm{sat}_{\mathcal{M}}(n,F)$ be the maximum rank of a weakly $F$-saturated matroid $M\in\mathcal{M}$ on the edges of $K_n$. Then, clearly, the analogue of \eqref{eq:asat_delete_vertex} holds for $\mathrm{rk}$-$\mathrm{sat}_{\mathcal{M}}(n,F)$ as well. And then the whole proof can be applied for this relaxed version of weakly saturation rank, implying
\begin{theorem}
For any subset-closed family of matroids $\mathcal{M}$ and any graph $F$, there exist $a_{\mathcal{M},F}\in\mathbb{Z}_{\geq 0}$ and $C_{\mathcal{M},F}\in\mathbb{Z}$ such that 
$$
\mathrm{rk}\text{-}\mathrm{sat}_{\mathcal{M}}(n,F)=a_{\mathcal{M},F}n+C_{\mathcal{M},F}
$$ 
for all large enough $n$.
\label{th:asymp_general_matroids}
\end{theorem}
In particular, for graphs $F$ and $G$, we call a vector space $W$ {\it weakly $F$-saturated in $G$}, if there exists a map $\varphi:E(G)\to W$ so that $\{\varphi(e),\,e\in E(G)\}$ spans the entire $W$ and every copy of $F$ in $G$ maps to a cycle in $W$. Given a field $\mathbb{F}$, let $\mathrm{rk}$-$\mathrm{sat}_{\mathbb{F}}(G,F)$ be the maximum dimension of a weakly $F$-saturated vector space $W$ over $\mathbb{F}$. We let $\mathrm{rk}$-$\mathrm{sat}_{\mathrm{linear}}(G,F)=\max_{\mathbb{F}}\mathrm{rk}$-$\mathrm{sat}_{\mathbb{F}}(G,F).$\footnote{For an {\it infinite} $\mathbb{F}$, the condition that every copy $F'$ of $F$ is a cycle is equivalent to the following: there exist non-zero $\lambda_e\in\mathbb{F}$, $e\in E(F')$, so that $\sum_{e\in E(F')}\lambda_e\varphi(e)=0$.  For vector spaces over finite fields this condition is stronger than the requirement that every $F$ is a cycle. Since in the definition of $\mathrm{rk}$-$\mathrm{sat}_{\mathrm{linear}}(G,F)$ we maximise the dimension over all vector spaces and since any vector space over a finite field can be extended to a vector space over an infinite field, the saturation rank can be defined as the maximum dimension of a vector space with the above stronger condition.}
\begin{corollary}
For every graph $F$ and every field $\mathbb{F}$, there exist $\tilde a_F,\tilde a_{\mathbb{F};F}\in\mathbb{Z}_{\geq 0}$ and $\tilde C_F,\tilde C_{\mathbb{F};F}\in\mathbb{Z}$ so that 
$$
\mathrm{rk}\text{-}\mathrm{sat}_{\mathbb{F}}(K_n,F)=\tilde a_{\mathbb{F};F}n+\tilde C_{\mathbb{F};F},\quad
\mathrm{rk}\text{-}\mathrm{sat}_{\mathrm{linear}}(K_n,F)=\tilde a_{F}n+\tilde C_F
$$ 
for all large enough $n$.
\label{cor:linear}
\end{corollary}

This corollary can be used to show a moderately explicit description of a vector space spanned by $E(K_n)$ at which the maximum dimension $\mathrm{rk}$-$\mathrm{sat}_{\mathbb{F}}(K_n,F)$ is achieved. In particular, such vector spaces were used in~\cite{Kalai85,KMM} to prove tight lower bounds for $\mathrm{wsat}(K_n,K_s)$ and $\mathrm{wsat}(K_n,K_{s,t})$. In what follows, for the sake of simplicity of presentation, for a weakly $F$-saturated vector space $W$, we omit the map $\varphi$ and associate edges in $G$ with vectors from $W$.

\begin{theorem}
Let $F$ be a graph on $v$ vertices, $\mathbb{F}$ be a field, and let $a,n_0\in\mathbb{Z}_{\geq 0},$  $C\in\mathbb{Z}$ be such that, for every integer $n\geq n_0$, $\mathrm{rk}$-$\mathrm{sat}_{\mathbb{F}}(K_n,F)=a(n-v)+C.$ Then, for every weakly $F$-saturated in $K_n$ vector space $W$ over $\mathbb{F}$ with $\mathrm{dim}W=\mathrm{rk}$-$\mathrm{sat}_{\mathbb{F}}(n,F)$, there exist vectors $f_u\in\mathbb{F}^{n-v-C}$, $u\in[n]$ and $g_{(w,u)}\in\mathbb{F}^a$, $h_{(u,w)}\in\mathbb{F}^{C(a+1)}$, $u\neq w\in[n]$, such that 
$$
\{u,w\}=(f_u\otimes g_{(w,u)}+f_w\otimes g_{(u,w)})\oplus h_{(u,w)},\quad \text{over all }u\neq w\in[n].
$$
\end{theorem}

\begin{proof}
Due to Corollary~\ref{cor:linear}, there exist $a,n_0\in\mathbb{Z}_{\geq 0}$ and $C\in\mathbb{Z}$ such that, for all $n\geq n_0-1$, $\mathrm{rk}$-$\mathrm{sat}_{\mathbb{F}}(K_n,F)=a(n-v)+C=:d$. Fix $n\geq n_0$ and consider a map $\varphi:E(K_n)\to \mathbb{F}^{d}$ such that $\langle\varphi(e),\,e\in E(K_n)\rangle=\mathbb{F}^{d}$, every copy of $F$ in $K_n$ maps to a cycle in $\mathbb{F}^{d}$ (again, everywhere below in the proof $\varphi$ is omitted). 

Let us consider the weakly $F$-saturated (in $K_n$) graph $H$ containing all edges adjacent to a vertex from $[v]$. By Claim~\ref{cl:satgraph-preserve-rank}, $\mathbb{F}^{d}$ has a basis $B\subset E(H)$. Assume that some $u\notin[v]$ belongs to at most $a-1$ edges that have an image in $B$. Then the linear matroid spanned by all edges from $E(H)$ that do not touch $u$ has rank at least $|B|-(a-1)=a(n-1-v)+C+1>\mathrm{rk}$-$\mathrm{sat}_{\mathbb{F}}(K_{n-1},F)$ --- contradiction with Claim~\ref{cl:satgraph-preserve-rank}. 

So, we get that every vertex outside $[v]$ belongs to at least $a$ edges that have an image in $B$. Since the total number of such edges equals $a(n-v)+C$, we get that there are at least $n-v-C$ vertices outside $[v]$ that belong to exactly $a$ edges that are mapped into $B$. Without loss of generality, we assume that all vertices from $[n-C]\setminus[v]$ are such. Let $S=[v]\sqcup([n]\setminus[n-C])$. Note that all edges spanned by $S$ form a basis in $\langle K_n[S]\rangle$. Indeed, otherwise, let $\tilde B\supset B[S]$ be a basis in $\langle K_n[S]\rangle$. Then any basis in $\mathbb{F}^{d}$ consisting of $\varphi(e)$ for some $e\in E(K_n)$ that contains all vectors from $\tilde B$ must have a vertex $u\in[n]\setminus S$ such that there are at most $a-1$ edges touching $u$ and that are mapped into $\tilde B$. In a similar way, this leads to a contradiction with Claim~\ref{cl:satgraph-preserve-rank}. Thus, $\langle K_n[S]\rangle\cong\mathbb{F}^{C(a+1)}$.

For $u\in[n]\setminus S$ and $i\in[a]$, let $\nu_i(u)$ be the $i$-th neighbour of $u$ in $[v]$ such that $\{u,\nu_i(u)\}\in B$. Note that, for any $x\in[v]$, $\{u,x\}$ belongs to the vector space spanned by $\{u,\nu_i(u)\}\cup H[S]$ since otherwise some vector from $B$ that touches a vertex outside $S\cup\{u\}$ can be replaced with $\{u,x\}$, leading again to a contradiction with Claim~\ref{cl:satgraph-preserve-rank}.  Consider a copy of $F$ in $K_n$ that contains vertices $u,w\notin S$ and $v-2$ vertices from $[v]$. Since this copy is a cycle, $\{u,w\}\in\langle\{u,x\in [v]\}\cup\{w,x\in [v]\}\cup K_n[[v]]\rangle$, implying $\{u,w\}\in\langle\{u,\nu_i(u)\}\cup\{w,\nu_i(w)\}\cup K_n[S]\rangle$. In a similar way, for every $u\notin S$, $w\in S$, the vector $\{u,w\}$ belongs to the vector space spanned by $\{u,\nu_i(u)\}$ and $K_n[S]$. It remains to observe that we may assign independent vectors $f_u\in\mathbb{F}^{n-C-v}$ to $u\notin S$ in a way such that the vector space spanned by $\{u,\nu_i(u)\}$ over all $u\in [n]\setminus S$ and $i\in[a]$ is isomorphic to $\mathbb{F}^{n-C-v}\otimes\mathbb{F}^{a}$ due to an isomorphism that maps $\{u,\nu_i(u)\}$ to $f_u\otimes f_i(u)$ for some $f_i(u)\in\mathbb{F}^{a}$.
\end{proof}

\begin{remark}
In~\cite{Kalai85}, Kalai considered symmetric matroids and studied their ranks. A matroid $M$ on ${\mathbb{N}\choose 2}$ is {\it symmetric}, if the set of its circuits is invariant under the action of the symmetric group over $\mathbb{N}$. Let $M_n$ be the restriction of $M$ into ${[n]\choose 2}$. Kalai proved that $b_{M}:=\lim_{n\to\infty}\frac{\mathrm{rk}M_n}{n}$ exists and it is integer whenever $M$ is non-trivial in the sense that there exists $n$ such that ${[n]\choose 2}$ is not independent in $M$. Kalai observed an interesting relation between $b_M$ and circuits of $M$: $\mathrm{rk}M_n={n\choose 2}-{n-b_M\choose 2}$ for every $n\geq b_M$ if and only if $K_{b_M+2}$ is a circuit in $M$.

We claim that Theorem~\ref{th:asymp_general_matroids} immediately implies that $b_M$ exists and that $b_M\in\mathbb{Z}_{\geq 0}$. Thus (1) it gives an alternative proof of the assertion proved by Kalai and (2) it is stronger than the existence of an integer limit of the scaled rank of a symmetric matroid. Indeed, let $M$ be a symmetric non-trivial matroid on ${\mathbb{N}\choose 2}$. Consider the family $\mathcal{M}=\{M_n,\,n\in\mathbb{Z}_{>0}\}$ that comprises all restrictions of $M$ into ${[n]\choose 2}$. Let $n_0$ be such that ${[n_0]\choose 2}$ is not independent, and let $F$ be a graph on $[n_0]$ such that $E(F)$ is a circuit in $M_{n_0}$. Due to symmetry and the fact that any circuit is also a cycle, we get that, for every $n\geq n_0$, $M_n$ is weakly $F$-saturated, and then  Theorem~\ref{th:asymp_general_matroids} applies: $b_M=a_{\mathcal{M},F}\in\mathbb{Z}_{\geq 0}$.
\end{remark}

\section{New method}
\label{sc:new-method}

In this section we suggest a simple modification of Lemma~\ref{lm:Kalai} and show that it can be used to get the exact value of $\mathrm{wsat}(K_n,F)$ even when $c_F$ is not integer.\\

We define the weak saturation number and the weak saturation rank for multigraphs in literally the same way as for simple graphs, though we need a slightly more general concept --- weak saturation with respect to a family of graphs. Fix a multigraph $G$ (the set of edges of a {\it multigraph} on the set of vertices $V$ is a multiset of pairs from ${V\choose 2}$) and a family of multigraphs $\mathcal{F}$. The {\it weak $\mathcal{F}$-saturation number} of $G$ (denoted by $\mathrm{wsat}(G,\mathcal{F})$) is the minimum number of edges in a {\it weakly $\mathcal{F}$-saturated submultigraph of $\mathcal{F}$}, that is, in a multigraph $H$ such that $G$ can be obtained from it by adding missing edges (with multiplicity 1) one by one, each time creating a copy of some $F\in\mathcal{F}$. A matroid $M=(E(G),\mathcal{I})$ on the edges of the mulitgraph $G$ (note that different instances of the same  element in the multiset $E(G)$ are different elements in the ground set $E(G)$ of $M$) is {\it weakly $\mathcal{F}$-saturated}, if, for every $F\in\mathcal{F}$, every copy $\tilde F$ of $F$ in $G$ is a {\it cycle} in $M$. The {\it weak $\mathcal{F}$-saturation rank} of $G$ (denoted by $\mathrm{rk}$-$\mathrm{sat}(G,\mathcal{F})$) is the maximum rank of a weakly $\mathcal{F}$-saturated matroid on $E(G)$. The analogue of Lemma~\ref{lm:Kalai} holds for multigraphs as well. Let us prove it for the sake of completeness.

\begin{lemma}
For every multigraph $G$ and every family of multigraphs $\mathcal{F}$, we have that $\mathrm{wsat}(G,\mathcal{F})\geq\mathrm{rk}$-$\mathrm{sat}(G,\mathcal{F})$.
\label{lm:Kalai_multi}
\end{lemma}

\begin{proof}
Let $H$ be a weakly $\mathcal{F}$-saturated multigraph and let $M$ be a weakly $\mathcal{F}$-saturated matroid. It is sufficient to prove that $\mathrm{rk}M=\mathrm{rk}_M(E(H))\leq|E(H)|$. 
 Consider a sequence of multigraphs $H=H_0\subset H_1\subset\ldots\subset H_m=G$, where $E(H_i)\setminus E(H_{i-1})$ is a single (instance of) edge $e_i$ and $H_i$ contains a copy $F_i\cong F$ for some $F\in\mathcal{F}$ that, in turn, contains $e_i$. Let us show that, for every $i\geq 1$, $\mathrm{rk}_M(E(H_i))=\mathrm{rk}_M(E(H_{i-1}))$ --- this would complete the proof. Since $E(F_i)$ is a cycle, it contains an independent set $I$ of size $\mathrm{rk}_M(E(F_i))$ that does not contain $e_i$. So, $I\cup\{e_i\}$ is not independent. Let us extend it to a basis set $B\supset I$ of the matroid induced on $H_i$. Since it does not contain $e_i$ (otherwise,  $I\cup\{e_i\}\subset B$ must be independent), we get that $B$ is entirely in $H_{i-1}$, and so $|B|=\mathrm{rk}_M(E(H_{i-1}))$,~as~needed.
\end{proof}

Let us now introduce the method. Consider simple graphs $F$ and $G$. Let $G^k$ be a multigraph obtained from $G$ by including $k$ instances of every edge of $G$. For every, $e\in E(F)$ let us also consider a multigraph $F^k_e$ that contains each edge from $E(F)\setminus e$ exactly $k$ times and the edge $e$ --- once. Then
\begin{equation}
\mathrm{rk}\text{-}\mathrm{sat}\left(G^k,\{F^k_e,\,e\in E(F^k)\}\right)\leq
\mathrm{wsat}\left(G^k,\{F^k_e,\,e\in E(F^k)\}\right)\leq k\cdot\mathrm{wsat}(G,F).
\label{eq:new-method}
\end{equation} 
Indeed, the first inequality holds due to Lemma~\ref{lm:Kalai_multi}, while the second inequality immediately follows from the observation that, if $H$ is weakly $F$-saturated in $G$, then $H^k$ has $k|E(H)|$ edges and it is weakly $\{F^k_e,\,e\in E(F)\}$-saturated in $G^k$.~We~get
\begin{lemma}
 For any two simple graphs $G,F$, 
 $$
 \mathrm{wsat}(G,F)\geq\frac{1}{k}\cdot\mathrm{rk}\text{-}\mathrm{sat}\left(G^k,\{F^k_e,\,e\in E(F)\}\right).
 $$
\label{lm:rational-from-multi}
\end{lemma}

Let us finish this section with an example of application of Lemma~\ref{lm:rational-from-multi}. Let us prove that, for $n$ divisible by $k\geq 3$, $\mathrm{wsat}(G=K_n,F=B_k)=\frac{k-1}{2}n$, where $B_k$ is a {\it dumb-bell} consisting of two disjoint $k$-cliques joined by a single edge connecting them. We have to note that combinatorial proofs of this equality are known~\cite{PikhurkoPhD,TZ}. Moreover, Pikhurko noted in~\cite{PikhurkoPhD} that, for odd $k$, this equality can be derived from Lemma~\ref{lm:Kalai_multi} by considering count matroids. Here we show that count matroids are also useful in the case of even $k$, but we shall defined them on multigraphs and apply Lemma~\ref{lm:rational-from-multi}. Let us notice that count matroids on edges of multigraphs were considered in the literature, see, e.g.,~\cite{LeeStreinu}.

The upper bound is straightforward --- a disjoint union of $n/k$ $k$-cliques is weakly $B_k$-saturated in $K_n$. For the lower bound, let us consider a matroid $M$ on $E(G^k)$ with independent sets $I$ satisfying $|I'|\leq v(I'){k\choose 2}$ for every $I'\subseteq I$, where $v(I')$ is the number of vertices touching edges from $I'$. The fact that $M$ is a matroid follows from submodularity of $v: 2^{[n]\choose 2}\to\mathbb{Z}_{\geq 0}$ and the fact that any submodular non-decreasing integer-valued function $\rho$ on the set of all subsets of a finite set $X$ defines a matroid on $X$ with circuits that satisfy $|C|\geq\rho(C)+1$ while $|C'|\leq\rho(C')$ for every proper $C'\subset C$ (see~\cite{ER-matroids},~\cite[Proposition 12.1.1]{Oxley}). It is easy to check directly that all copies of $F_e^k$, $e\in E(F^k)$, are cycles in $M$, and the union of $n/k$ disjoint $K_k^k$ is an independent set in $M$ implying the desired lower bound by Lemma~\ref{lm:rational-from-multi}.

\begin{remark}
Lemma~\ref{lm:rational-from-multi} implies the value of $\mathrm{wsat}(K_n,B_k)$ for all $n$ up to a constant additive term, that depends only on $k$. It can be further generalised to all $F$ comprising disjoint cliques of different sizes joined by a single edge and some other graphs $F$ with `weak' connectivity properties.
\end{remark}


\section{Weak saturation rank of random graphs}
\label{sc:random}

Let $p=\mathrm{const}\in(0,1)$, $G_n\sim G(n,p)$. In~\cite{KMMT-R}, it is proven that, for any $F$, whp 
\begin{equation}
\mathrm{wsat}(G_n,F)=(1+o(1))\mathrm{wsat}(K_n,F).
\label{eq:asympt_random_wsat}
\end{equation}

\begin{remark} 
For every positive integer $v$, whp, for any pair of vertices $x_1,x_2$, there exists a $v$-clique $F'$ in $K_n$ containing $x_1,x_2$ such that the spanning subgraph $F'\setminus\{x_1,x_2\}\subset F'$ obtained by the removal of the edge $\{x_1,x_2\}$ is also a subgraph of $G_n$ (see, e.g.,~\cite{Spencer}). Thus, whp $\mathrm{wsat}(G_n,F)\geq\mathrm{wsat}(K_n,F)$. In~\cite{Kal_Zhuk}, it was conjectured that whp $\mathrm{wsat}(G_n,F)=\mathrm{wsat}(K_n,F)$ for all $F$. This conjecture is known to be true, in particular, for cliques and complete bipartite graphs.
\end{remark}

As we noted in Section~\ref{sc:sat-rank}, for every odd integer $k\geq 3$, there exists a connected graph $F$ with $\mathrm{wsat}(K_n,F)=\frac{k}{2}n+O(1)$.
The fact that, for infinitely many connected $F$ there exists $\varepsilon>0$ such that whp $\mathrm{rk}\text{-}\mathrm{sat}(G_n,F)<(1-\varepsilon)\mathrm{wsat}(G_n,F)$ follows immediately from~\eqref{eq:asympt_random_wsat} and the following result.


\begin{theorem}\label{th:asat-integer-random}
Let $p=\mathrm{const}\in(0,1)$, $G_n\sim G(n,p)$. Let $F$ be a graph with $\mathrm{wsat}(K_n,F)=(c_F+o(1))n$ for some constant $c_F\geq 0$. Then whp $\mathrm{rk}$-$\mathrm{sat}(G_n,F)\leq \lfloor c_F\rfloor n+(\log n)^{O(1)}$.
\end{theorem}

To prove Theorem~\ref{th:asat-integer-random}, we need the following lemmas.

\begin{lemma}[\cite{KMMT-R}\footnote{Although this lemma is not stated in \cite{KMMT-R} explicitly, the proof of \cite[Theorem 1.3]{KMMT-R} is essentially the proof of this lemma. Indeed, authors prove that as soon as a graph on $[n]$ has properties from \cite[Lemma 2.2]{KMMT-R} and \cite[Lemma 2.3]{KMMT-R}, the graph has the weak saturation number at most $\mathrm{wsat}(K_n,F)+o(n)$, and the properties from \cite[Lemma 2.2]{KMMT-R} and \cite[Lemma 2.3]{KMMT-R} are to admit a clique factor and to contain a large common neighbourhood for every set of a given size.}]
Let $\varepsilon>0$, $\delta>0$, and let $F$ be a graph on $v$ vertices. There exists $C=C(\delta,\varepsilon,v)>0$ such that, for every integer $n\geq C$ and every graph $G$ on $[n]$ satisfying
\begin{itemize}
\item $[n]$ admits a partition $[n]=V_1\sqcup\ldots\sqcup V_m$ such that each $V_i$ has size at least $C$ and induces a clique in $G$,
\item every $5v$ vertices in $G$ have at least $\delta n$ common neighbours,
\end{itemize}
the following holds: $\mathrm{wsat}(G,F)\leq\mathrm{wsat}(K_n,F)+\varepsilon n$.
\label{lm:iranian}
\end{lemma}

\begin{lemma}
Fix a graph $F$ on $v$ vertices satisfying $\mathrm{wsat}(K_n,F)=(c_F+o(1))n$. 
Whp, for every induced subgraph $H$ of $G_n$ of size at least $\ln^{5v+2}n$, $\mathrm{wsat}(H,F)\leq(c_F+o(1))|V(H)|$.
\label{lm:asymptotic-stability-strong}
\end{lemma}

\begin{proof}
Set $\delta=p^{5v}/2^{5v+1}$ and fix any $\varepsilon>0$. Let $C=C(\delta,\varepsilon,v)>0$ be the constant stated in Lemma~\ref{lm:iranian}. 

Fix a set $V\subset[n]$ of size $m\geq\ln^{5v+2} n$.
For $s\in\{0,1,\ldots,n\}$, let us call an $s$-set $\mathbf{x}\subset[n]$ {\it bad} with respect to $V$, if it has less than $(p/2)^s m$ common neighbours in $V$. Otherwise, the set is {\it good}. An empty set ($s=0$) is always good. 
  Now let us fix $(\ln n\ln\ln n)^t$ $t$-sets. They occupy at least $\ln n\ln\ln n$ vertices. Therefore, we may find a sequence of at least $\frac{1}{t}\ln n\ln\ln n$ $t$-sets $\mathbf{y}_1,\ldots,\mathbf{y}_g$ among them such that each $\mathbf{x}_i=\mathbf{y}_i\cap(\mathbf{y}_1\cup\ldots\cup\mathbf{y}_{i-1})$ is a {\it proper} subset of $\mathbf{y}_i$. Let $\mathbf{x}_1=\mathbf{x}_{g+1}=\varnothing.$ Then 
\begin{multline*}
 \mathbb{P}(\text{all }\mathbf{y}_i\text{ are bad, all }\mathbf{x}_i\text{ are good}) \leq
\prod_{i=1}^{g}
  \mathbb{P}(\mathbf{y}_i\text{ is bad}\mid\mathbf{y}_{j\leq i-1}\text{ are bad, }\mathbf{x}_{j\leq i}\text{ are good}).
  \end{multline*}
Note that, if $\mathbf{y}_i$ is bad, then it has at most $(p/2)^t m$ common neighbours in the set $V\setminus(\mathbf{y}_1\cup\ldots\cup\mathbf{y}_{i-1})$, which has size $|V|-O(\ln n\ln\ln n)$. Subject to the event that $\mathbf{y}_1,\ldots,\mathbf{y}_{i-1}\text{ are bad}$ and $\mathbf{x}_1,\ldots,\mathbf{x}_i\text{ are good}$, probability that $\mathbf{y}_i$ has at most $(p/2)^t m$ common neighbours in $V\setminus(\mathbf{y}_1\cup\ldots\cup\mathbf{y}_{i-1})$ is at most $\exp(-\Omega(m))$ by the Chernoff bound. Therefore,
\begin{align*}
 \mathbb{P}(\text{all }\mathbf{y}_i\text{ are bad, all }\mathbf{x}_i\text{ are good}) 
  \leq
 \exp(-\Omega(m)\ln n\ln\ln n).
\end{align*}
By the union bound over $t\in[5v]$, over $V\subset[n]$, and over $(\ln n\ln\ln n)^t$-tuples of $t$-sets, we get that whp, for every set $V$ of size at least $\ln^{5v+2} n$ and every $t\in[5v]$, there are less than $(\ln n\ln\ln n)^t$ bad $t$-sets with respect to this set $V$ such that all their $(t-1)$-subsets are good. Assuming that this property (we call it {\it property of bad sets}) holds in a graph on $[n]$ deterministically, we may fix a set $V$ and then distinguish a set $V'_1\subset[n]$ comprising all bad 1-sets (they occupy less  than $\ln n\ln\ln n$ vertices). Then we may consider the union $V'_2$ of all bad 2-sets lying entirely outside $V'_1$. Clearly, $|V'_2|\leq 2(\ln n\ln\ln n)^2$. Proceeding inductively, we get that we may exclude at most $\sum_{i=1}^t i(\ln n\ln\ln n)^i$ vertices from the graph and get rid of all bad subsets up to size $t$. 

In what follows, we assume that the property of bad sets holds deterministically in $G_n$. We shall also note that, by Janson's inequality~\cite[Theorem 2.18]{Janson}, in a fixed set of size $m\geq\ln n$, there is a clique of size at least $C$ with probability $1-\exp(-\Omega(m^2))$. By the union bound over $V\subset[n]$, we get that whp, every  set $V$ of size at least $\ln n\ln\ln n$ has a clique of size at least $C$. We also assume that this property holds deterministically in $G_n$.

Let us finally fix $V\subset[n]$ of size at least $\ln^{5v+2}n$ and find a weakly $F$-saturated subgraph in $H:=G_n[V]$ with at most $(c_F+3\varepsilon)|V|$ edges. Let us remove a set of size at most $(5v)^2(\ln n\ln\ln n)^{5v}$ vertices from $V$ and get rid of all bad $t$-sets, $t\in[5v]$, with respect to $V$. 
 From the remaining set we can remove at most $\ln n\ln\ln n$ vertices and decompose the remainder (denoted by $\tilde H$) into cliques of size at least $C$. By Lemma~\ref{lm:iranian}, there exists a weakly $F$-saturated spanning subgraph $H'\subset \tilde H$ with at most $(c_F+\varepsilon)|V(\tilde H)|$ edges. If $x\in V\setminus V(\tilde H)$ has at least $\ln n\ln\ln n$ neighbours in $\tilde H$, then we find a clique of size at least $C$ in this neighbourhood in $\tilde H$ and retain $v$ edges from $x$ to this clique. Otherwise, we retain all edges from $x$ to $\tilde H$. We get a weakly $F$-saturated subgraph of $G_n[V]$ with at most
$$
 (c_F+2\varepsilon)|V(\tilde H)|+\ln n\ln\ln n((5v)^2(\ln n\ln\ln n)^{5v}+\ln n\ln\ln n)<(c_F+3\varepsilon)|V|
$$
edges, completing the proof.







\end{proof}

\begin{lemma}
Fix an integer $v\geq 3$. Let $X=\left[\left\lfloor \frac{6}{p^2}\ln n\right\rfloor\right]$. Whp, for every two different vertices $x_1,x_2\in[n]$, in their common neighbourhood $N_X(x_1,x_2)$ in $X$ there exists a clique of size $v$.
\label{lm:star-like-choice}
\end{lemma}

\begin{proof}
Fix different $x_1,x_2\in[n]$. By the Chernoff bound, with probability $1-o(n^{-2})$, $|N_X(x_1,x_2)|\geq\ln n$. Thus, we may expose edges between $x_1,x_2$ and the rest of the graph and assume that  $|N_X(x_1,x_2)|\geq\ln n$ holds deterministically. By Janson's inequality~\cite[Theorem 2.18]{Janson}, there is a $v$-clique in $|N_X(x_1,x_2)|$ with probability $1-\exp(-\Omega(\ln^2n))$. The union bound over the choice of $x_1,x_2$ completes the proof.
\end{proof}

\begin{proof}[Proof of Theorem~\ref{th:asat-integer-random}.]
Let us assume that the conclusions of Lemmas~\ref{lm:asymptotic-stability-strong}~and~\ref{lm:star-like-choice} hold deterministically in $G_n$. Assume that $M_n =\left(E(G_n),\mathcal{I}_n\right)$ is a weakly $F$-saturated matroid with the maximum rank. 
Let us consider the spanning subgraph $H\subset G_n$ containing all edges adjacent to a vertex from $X=\left[\left\lfloor \frac{6}{p^2}\ln n\right\rfloor\right]$. Due to the conclusion of Lemma~\ref{lm:star-like-choice}, $H$ is weakly $F$-saturated in $G_n$. By Claim~\ref{cl:satgraph-preserve-rank}, $M_n$ has a basis $B\subset E(H)$. Let us choose  a constant $\varepsilon>0$ small enough (it is enough to satisfy $\varepsilon(1+c_F+\varepsilon)<\lfloor c_F\rfloor+1-c_F$). We let $v$ be the number of vertices in $F$. 

Assume that there are at least $(\ln n)^{5v+2}$ vertices $x_1,\ldots,x_m\in [n]\setminus X$ that are adjacent to at least $\lfloor c_F\rfloor +1$ edges of $B$. Let us consider subgraphs $B'\subset B$ and $G'\subset G_n$ induced by $X\sqcup\{x_1,\ldots,x_m\}$. Letting $K'$ to be the clique on $V(G')=V(B')$, we get, by the conclusion of Lemma~\ref{lm:asymptotic-stability-strong}, 
$$
\mathrm{rk}\text{-}\mathrm{sat}(G',F)\leq\mathrm{wsat}(G',F)\leq (c_F+\varepsilon)|V(G')|\leq
(c_F+\varepsilon)(1+\varepsilon)m.
$$
On the other hand, since taking induced subgraphs, preserves~\eqref{eq:sat-rank-def},
$$
 \mathrm{rk}\text{-}\mathrm{sat}(G',F)\geq\mathrm{rk}_{M_n}(E(G'))\geq|B'|\geq m(\lfloor c_F\rfloor +1)
$$
--- a contradiction. 
 Therefore, 
$$
 |B|\leq{|X|\choose 2}+(\ln n)^{5v+2}|X|+n\lfloor c_F\rfloor=n\lfloor c_F\rfloor+O(\log^{5v+3} n),
$$
completing the proof.
\end{proof}

\section{Weak saturation rank of other discrete structures}
\label{sc:hyper}

Weak saturation numbers were also studied for other host graphs~\cite{BTT,KMM,MS} as well as for other structures --- for certain directed graphs (see, e.g.,~\cite{Alon85,BTT}) and for hypergraphs (see, e.g.,~\cite{MS,PikhurkoPhD}). In many cases, the respective analogues of Lemma~\ref{lm:Kalai} give tight lower bounds for these structures as well. In particular, the linear matroid that proves the tight lower bound for $\mathrm{wsat}(K_n,K_s)$ suggested by Kalai in~\cite{Kalai85}, can be directly generalised to hypergraphs implying $\mathrm{wsat}(K^{(r)}_n,K^{(r)}_s)={n\choose r}-{n-s+r\choose r}$  for $r$-uniform complete hypergraphs. Tight lower bounds for complete directed multipartite hypergraphs  as well as for complete undirected multipartite hypergraphs with balanced $F$ were proven in~\cite{BTT} using an analogue of Lemma~\ref{lm:Kalai} and a `colourful' modification of a linear matroid suggested by Kalai in~\cite{Kalai85}. Pikhurko in~\cite{PikhurkoPhD} used a reformulation of Lemma~\ref{lm:Kalai} for hypergraphs to prove tight lower bounds on weak saturation numbers for complete $G$ and specific hypergraphs $F$ called pyramids as well as for certain directed hypergraphs.

The definition of the weak saturation number generalises to hypergraphs in a straightforward way: for hypergraphs $G,F$, the {\it weak $F$-saturation number} of $G$ --- denoted by  $\mathrm{wsat}(G,F)$ --- is the minimum number of edges in a spanning subhypergraph of $G$ so that the missing edges of $G$ can be recovered one be one, creating a copy of $F$ at every step. For hypergraphs, Lemma~\ref{lm:Kalai} admits the following generalisation. As for graphs, we call a matroid $M=(E(G),\mathcal{I})$ on the (hyper)edges of $G$ {\it weakly $F$-saturated}, if every copy of $F$ in $G$ is a cycle in $M$; the {\it weak $F$-saturation rank}  $\mathrm{rk}$-$\mathrm{sat}(G,F)$ of the hypergraph $G$ is the maximum rank of a weakly $F$-saturated matroid $M$ on the (hyper)edges of $G$.

\begin{lemma}
For any hypergraphs $G,F$, $\mathrm{wsat}(G,F)\geq\mathrm{rk}$-$\mathrm{sat}(G,F)$.
\label{lm:Kalai_hypergraphs}
\end{lemma}

In this section, we show that for every $r\geq 3$, for infinitely many $r$-uniform $F$, Lemma~\ref{lm:Kalai_hypergraphs} cannot be used to get even the asympotics (in $n$) of  $\mathrm{wsat}(K^{(r)}_n,F)$. Indeed, we show below that, for any $r$-uniform hypergraph $F$, there exists a non-negative {\it real} $a_F^{(r)}$ such that
\begin{equation}
\lim_{n\to\infty}\frac{\mathrm{rk}\text{-}\mathrm{sat}(K^{(r)}_n,F)}{n^{s(F)-1}}=a^{(r)}_F,
\label{eq:hypergaphs_limit}
\end{equation}
and $a_F^{(r)}\in\mathbb{Z}_{\geq 0}$ if $s(F)\leq 2$, where $s(F)$ is the {\it sharpness} of $F$, i.e. the minimum non-negative integer so that, for some $e\in E(F)$, some $S\subset e$ of size $s(F)$, and any $e'\in E(F)\setminus\{e\}$, $S\not\subset e'$. On the other hand, for every $r\geq 3$, it is easy to construct infinitely many hypergraphs $F$ such that $\lim_{n\to\infty}\mathrm{wsat}(K^{(r)}_n,F)/n^{s(F)-1}\in(0,1)$ while $s(F)=2$, see Appendix. We finally recall that, for every $r$-uniform $F$, $\lim_{n\to\infty}\mathrm{wsat}(K^{(r)}_n,F)/n^{s(F)-1}$ exists due to the recently resolved conjecture of Tuza~\cite{ST_Tuza}.

\begin{theorem}
For every $r$-uniform hypergraph $F$ there exists $a_F^{(r)}\in\mathbb{R}_{\geq 0}$ such that~\eqref{eq:hypergaphs_limit} holds. If $s(F)\leq 2$, then $a_F^{(r)}\in\mathbb{Z}$ and there exists $C_F^{(r)}\in\mathbb{Z}$ such that 
$$
\mathrm{rk}\text{-}\mathrm{sat}(K^{(r)}_n,F)=a_F^{(r)}n+C_F^{(r)}
$$ 
for all large enough $n$.
\label{th:hyper}
\end{theorem}

\begin{proof}
Let $v=|V(F)|$, $s=s(F)$, and $f(n)=\mathrm{rk}$-$\mathrm{sat}(K^{(r)}_n,F)$. Let $M_n$ be a weakly $F$-saturated matroid on $E(K^{(r)}_n)$ with rank $f(n)$. Without loss of generality, we assume that $V(K^{(r)}_n)=[n]$. Let $H\subset K^{(r)}_n$ consist of all edges $e\in E(K^{(r)}_n)$ such that  $|e\cap([n]\setminus[v])|\leq s-1.$ It is easy to see that $H$ is weakly $F$-saturated in $K^{(r)}_n$ (see details in~\cite{Tuza}): missing edges can be recovered in ascending order of the cardinality of their intersection with $U:=[n]\setminus[v]$. Indeed, each such edge creates a copy of $F$ such that $V(F)\cap U=e\cap U$ and $S\subset e\cap U$, where $e$ and $S$ witness the definition of the sharpness of $F$.

Due to the generalisation of Claim~\ref{cl:satgraph-preserve-rank} to hypergraphs, $M_n$ has a basis $B\subset E(H)$. Let a vertex $u\in U$ have minimum degree in $B$. Since every edge of $B$ has at most $s-1$ vertices of $U$, we get that $u$ has degree at most $f(n)(s-1)/(n-v)$. Since deletion of vertices preserves cycles, we get that
\begin{equation}
f(n-1)=\mathrm{rk}\text{-}\mathrm{sat}(K^{(r)}_{n-1},F)\geq\mathrm{rk}_{M_n}(E(K^{(r)}_n\setminus\{u\}))\geq 
f(n)-\left\lfloor\frac{f(n)(s-1)}{n - v}\right\rfloor.
\label{eq:hyper_recursion}
\end{equation}
In particular, if $s=1$, then $f$ is not increasing and has non-negative integer values, implying $f(n)=C_F^{(r)}$ for some constant $C_F^{(r)}\in\mathbb{Z}_{\geq 0}$ and all large enough $n$. If $s=2$, then we get~\eqref{eq:th1-limit}, implying $\mathrm{rk}\text{-}\mathrm{sat}(K^{(r)}_n,F)=a_F^{(r)}n+C_F^{(r)}$ for all large enough $n$ where $a_F^{(r)}\in\mathbb{Z}_{\geq 0}$ and $C_F^{(r)}\in\mathbb{Z}$ due to Claim~\ref{eq:general_concluding_claim}. 
 For $s\geq 3$,~\eqref{eq:hyper_recursion} implies
$$
 f(n-1)\geq f(n)\left(1-\frac{s-1}{n-v}\right) =f(n)\frac{n-v-s+1}{n-v}=f(n)\frac{{n-v-1\choose s-1}}{{n-v\choose s-1}}.
$$
We then get that $f(n)/{n-v\choose s-1}$ decreases in $n$. Therefore, the limit $a_F^{(r)}$ in~\eqref{eq:hypergaphs_limit} indeed exists.
\end{proof}

We can also show similar asymptotic behaviour for complete multipartite hypergraphs. Let $K_{n,d}^{(r)}$ be a balanced complete $d$-partite $r$-uniform hypergraph with parts of size $n$. Then, for every $d$-partite $r$-uniform hypergraph $F$, there exists $\lim_{n\to\infty}\mathrm{rk}\text{-}\mathrm{sat}(K_{n,d}^{(r)},F)/ n^{s(F)-1}$, which is integer when $s(F)\leq 2$. In particular, for graphs, 
\begin{equation}
\mathrm{rk}\text{-}\mathrm{sat}(K_{n,d}^{(2)},F)=a_{F,d} n+C_{F,d}
\label{eq:rk-sat_multi}
\end{equation}
for some constant $a_{F,d}\in\mathbb{Z}_{\geq 0}$ and $C_{F,d}\in\mathbb{Z}$ and all large enough $n$. Let us prove~\eqref{eq:rk-sat_multi}. The statement for hypergraphs is generalised in the same way as the argument in Theorem~\ref{th:asat-integer-1} was adopted to derive Theorem~\ref{th:hyper}.

\begin{theorem}
For every integer $d\geq 2$ and every $d$-partite graph $F$, there exists a non-negative integer $a_{F,d}$ and an integer $C_{F,d}$ such that~\eqref{eq:rk-sat_multi} holds for all large enough $n$.
\label{th:multi_graphs_rk-sat}
\end{theorem}

\begin{proof}
Let $v=|V(F)|$ and $f(n)=\mathrm{rk}$-$\mathrm{sat}(G=K_{n,d}^{(2)},F)$. It is sufficient to prove~\eqref{eq:th1-limit} due to Claim~\ref{eq:general_concluding_claim}.

Assume that $M_n$ is a weakly $F$-saturated matroid on $E(G)$ with the maximum rank $f(n)$. Fix $v$-subsets $U_1\subset V_1,\ldots,U_d\subset V_d$ of the $d$ parts $V_1,\ldots,V_d$ of $G$. Let us consider the spanning subgraph $H\subset G$ containing all edges adjacent to a vertex from $U_1\sqcup\ldots\sqcup U_d$. Since $H$ is weakly $F$-saturated in $G$, by Claim~\ref{cl:satgraph-preserve-rank}, $M_n$ has a basis $B\subset E(H)$. Let a tuple $u_1\in V_1\setminus U_1,\ldots,u_d\in V_d\setminus U_d$ be such that the number of edges in $B$ having at least one vertex in $\{u_1,\ldots,u_d\}$ is minimum possible, denote this minimum number of edges by $\mu$. Let us count the number of pairs $(\mathbf{u}\in V_1\setminus U_1\times\ldots\times V_d\setminus U_d,e\in B)$ so that $e\cap\mathbf{u}\neq\varnothing.$ On the one hand, this number is at least $\mu(n-v)^d$. On the other hand, since every vertex from $V(G)\setminus(U_1\sqcup\ldots\sqcup U_d)$ participates in $(n-v)^{d-1}$ tuples $\mathbf{u}$, every edge from $B$ contains at most one vertex of $V(G)\setminus(U_1\sqcup\ldots\sqcup U_d)$, and the total number of edges in $B$ equals $f(n)$, we get $\mu(n-v)^d\leq f(n)(n-v)^{d-1}$, implying  $\mu\leq\left\lfloor f(n)/(n - v)\right\rfloor$.

It remains to delete vertices $u_1,\ldots,u_d$ from $G$ and observe that
$$
\mathrm{rk}\text{-}\mathrm{sat}(K_{n-1,d}^{(2)},F)\geq\mathrm{rk}_{M_n}(E(G\setminus\{u_1,\ldots,u_d\}))\geq \mathrm{rk}M_n-\left\lfloor\frac{f(n)}{n - v}\right\rfloor,
$$
implying~\eqref{eq:th1-limit} and completing the proof.
\end{proof}

In a similar way analogous results can be obtained for the weak saturation rank of directed multipartite graphs. Let us recall that the directed version of weak saturation numbers was studied in~\cite{Alon85,BTT,PikhurkoPhD}: Let $\overrightarrow{K}_{n,d}$ be a directed complete $d$-partite graph with $d$ labelled parts of size $n$ where edges are oriented from parts with smaller labels to parts with larger labels. Let $F$ be a directed subgraph of $\overrightarrow{K}_{n,d}$. Then the weak saturation number $\mathrm{wsat}\left(\overrightarrow{K}_{n,d},F\right)$ --- the minimum number of edges in a spanning subgraph of $\overrightarrow{K}_{n,d}$ such that the missing edges can be added one by one, each time creating a directed copy of $F$ --- is known for all complete directed $d$-partite $F$~\cite{Alon85,BTT} and it coincides with the weak $F$-saturation rank of $\overrightarrow{K}_{n,d}$. The {\it weak $F$-saturation rank}  $\mathrm{rk}$-$\mathrm{sat}\left(\overrightarrow{K}_{n,d},F\right)$ is the maximum rank of a matroid $M=\left(E\left(\overrightarrow{K}_{n,d}\right),\mathcal{I}\right)$ such that every copy of $F$ in $\overrightarrow{K}_{n,d}$ is a cycle in $M$. Then, for every directed graph $F$, that is a subgraph of $\overrightarrow{K}_{n_0,d}$ for some $n_0$, there exist integers $\overrightarrow{a}_{F,d}$ and $\overrightarrow{C}_{F,d}$ such that $\mathrm{rk}\text{-}\mathrm{sat}\left(\overrightarrow{K}_{n,d},F\right)=\overrightarrow{a}_{F,d} n+\overrightarrow{C}_{F,d}$ for all large enough $n$. Since the proof is identical with the proof of Theorem~\ref{th:multi_graphs_rk-sat}, we omit it.
















\section{Open questions}
\label{sc:discussions}

In Section~\ref{sc:new-method}, we show that multigraphs can be used to get tight lower bounds for $\mathrm{wsat}(K_n,F)$ when $\lim_{n\to\infty}\mathrm{wsat}(K_n,F)/n$ is not integer due to Lemma~\ref{lm:rational-from-multi}. It is natural to ask, whether for any graph $F$ there exists $k\in\mathbb{Z}_{>0}$ such that 
$$
\mathrm{wsat}(K_n,F)=\frac{1}{k}\cdot\mathrm{rk}\text{-}\mathrm{sat}\left(K_n^k,\{F^k_e,\,e\in E(F)\}\right).
$$

For random graphs, we are only able to prove that whp 
$$
\mathrm{rk}\text{-}\mathrm{sat}(G(n,p),F)\leq\left(\left\lfloor\frac{\mathrm{wsat}(K_n,F)}{n}\right\rfloor+o(1)\right) n
$$
for any constant $p\in(0,1)$ (see Section~\ref{sc:random}). Nevertheless, we believe that, for every $F$, the limit in probability of $\mathrm{rk}\text{-}\mathrm{sat}(G(n,p),F)/n$ exists and it is integer. It would be also interesting to know whether this limit equals $a_F=\lim_{n\to\infty}\mathrm{rk}\text{-}\mathrm{sat}(K_n,F)/n$ --- notice that the respective stability result for weak saturation numbers~\eqref{eq:asympt_random_wsat} is known~\cite{KMMT-R}. 

For $r$-uniform hypergraphs $F$ with $s(F)=2$, we prove in Section~\ref{sc:hyper} that the limit $\lim_{n\to\infty}\mathrm{rk}$-$\mathrm{sat}(K_n^{(r)},F)/n$ is integer and that the second-order term is $O(1)$. For larger $s(F)$, we only manage to prove that $\lim_{n\to\infty}\mathrm{rk}$-$\mathrm{sat}(K_n^{(r)},F)/n^{s(F)-1}$ exists. Thus, still we cannot prove that for some $F$ with $s(F)>2$, $\mathrm{rk}$-$\mathrm{sat}(K_n^{(r)},F)$ is strictly smaller than $\mathrm{wsat}(K_n^{(r)},F)$. It would be interesting to find such examples and to understand better the limiting behaviour of  $\lim_{n\to\infty}\mathrm{rk}$-$\mathrm{sat}(K_n^{(r)},F)$ for such~$F$.

In \cite{JT_val}, Jackson and Tanigawa defined a function $\mathrm{val}_{\mathcal{F}}:2^{E(K_n)}\to\mathbb{N}$ for a family $\mathcal{F}=\{F_1,F_2,\ldots\}$ of graphs in the following way (though the definition given in~\cite{JT_val} is different, the definition below is equivalent):
$$
\forall H\subset K_n\quad \mathrm{val}_{\mathcal{F}}(H) = \min_{H\subset G\subset K_n}\mathrm{wsat}(G, \mathcal{F}).
$$
They conjectured that if, for a given $\mathcal{F}$, there exists a unique maximal ($M_1\leq M_2$ if every every independent set in $M_1$ is independent in $M_2$) matroid $M$ on $E(K_n)$ such that each copy of every $F_i\in\mathcal{F}$ is its cycle ($M$ is {\it $\mathcal{F}$-cyclic}), then, for every $H\subset K_n$, $\mathrm{rk}_M(H)=\mathrm{val}_{\mathcal{F}}(H)$. Note that if $ \mathrm{val}_{\mathcal{F}}(H)$ is indeed the rank function of a matroid, then all copies of graphs from $\mathcal{F}$ are cycles for this matroid. Moreover, if this is the case,
$\mathrm{rk}\text{-}\mathrm{sat}(K_n,\mathcal{F})=\mathrm{val}_{\mathcal{F}}(K_n)=\mathrm{wsat}(K_n,\mathcal{F})$ due to the definition of $\mathrm{val}_{\mathcal{F}}$. Since, as we show in the paper, there exist singleton $\mathcal{F}$ such that the above equality does not hold, therefore $\mathrm{val}_{\mathcal{F}}$ is not the rank function of a matroid, and these $\mathcal{F}$ are good candidates to disprove the conjecture: if for such an $\mathcal{F}$ there are no two maximal $\mathcal{F}$-cyclic matroids, then the conjecture is false.

\section*{Acknowledgements}

The authors would like to thank Imre Leader for helpful comments on the paper.


\bibliographystyle{amsplain}

\section*{Appendix: a hypergraph with a non-integer limit}

Fix $r\geq 3$. Let $v_1\geq 2r-1,v_2\geq r+1$ be arbitrary integers. The desired family of $r$-uniform hypergraphs $F=F(v_1,v_2)$ is defined as follows. First, we consider a hypergraph $F_0$ on $v_1+v_2$ vertices that is a disjoint union of a $v_1$-cycle and a $v_2$-clique (a {\it $v$-cycle} is isomorphic to the hypergraph on $[v]$ with all $v$ edges comprising $r$ consecutive numbers in the cyclic order $(1,\ldots,v)$). Let $e_1,\ldots,e_m$ be all the missing edges of a clique on $V(F_0)$ that do not belong to $E(F_0)$. For every $i\in[m]$, we define $F_i=F_0\cup\{e_1,\ldots,e_i\}$. The hypergraph $F$ is a disjoint union of graphs isomorphic to $F_0,F_1,\ldots,F_m$. Let $v=|V(F)|$. 

\begin{claim}
$s(F)=2$ and $\lim_{n\to\infty}\mathrm{wsat}(K_n^{(r)},F)/n\in(0,1)$.
\end{claim}

\begin{proof}
Let $F_0$ be the `0-component' of $F$, $C_0$ be its cycle, and let $e$ be any edge $C_0$. Let $S\subset e$ consist of two opposite vertices of $e$. It is easy to see that $e$ and $S$ witness the definition of the sharpness, thus $s(F)=2$. Indeed, since $v_1\geq 2r-1$, any other edge of the cycle has at most one common vertex with $S$, and all the other connected components of $F$ are disjoint with $S$. On the other hand, every vertex from $V(F)$ belongs to at least 2 edges of $F$.

We will prove that $\lim_{n\to\infty}\mathrm{wsat}(K_n^{(r)},F)/n\in(0,1)$ by proving upper and lower bounds on $\mathrm{wsat}(K_n^{(r)},F)$ and recalling that the limit exists due to~\cite{ST_Tuza}. First, it is easy to observe that $\mathrm{wsat}(K_n^{(r)},F)\geq\frac{\delta(F)-1}{r}n$ where $\delta(F)$ is the minimum degree of $F$ (see details in~\cite{Tuza}), thus $\mathrm{wsat}(K_n^{(r)},F)\geq\frac{1}{r}n$. It remains to prove an upper bound. Let us show that, for large enough $n$, $\mathrm{wsat}(K_{n+v_1}^{(r)},F)\leq\mathrm{wsat}(K_n^{(r)},F)+v_1-1$. It would imply $\mathrm{wsat}(K_n^{(r)},F)\leq\frac{v_1-1}{v_1}n+O(1)$, completing the proof.

Let $n\geq 2v$ and let $H\subset K_n^{(r)}$ be weakly $F$-saturated in $K_n^{(r)}$. Let us show that a disjoint union of $H$ with a copy of $C_0\setminus e$, where $e$ is an arbitrary edge of $C_0$, is weakly saturated in $K_{n+v_1}^{(r)}$. Indeed, first of all, all missing edges of $G\cong K_n^{(r)}\supset H$ can be recovered. Then, the edge $e$ can be recovered as well, since a copy of $F$ excluding $C_0$, the cycle of $F_0$, is contained in $G$. Let us finally show that, for every $v_2$-subset $U\subset V(G)$, all edges comprising vertices from $U\sqcup V(C_0)$ can be recovered. We first recover an edge $e'_1$ such that $C_0\cup K_U\cup e'_1$, where $K_U$ is the clique on $U$, is isomorphic to $F_1$. It is possible since a copy of $F\setminus F_1$ is contained in $G\setminus U$. In a similar manner, we can recover all the missing edges on $U\sqcup V(C_0)$.
\end{proof}

\end{document}